\documentclass[12pt]{amsart}
\usepackage{txfonts}
\usepackage{amscd,amssymb,mathabx,subfigure}
\usepackage[arrow,matrix,graph,frame,poly,arc,tips]{xy}
\usepackage{graphicx,calrsfs}
\usepackage[colorlinks,plainpages,backref,urlcolor=blue]{hyperref}
\usepackage{tikz}
\usetikzlibrary{decorations.markings}
\usetikzlibrary{shapes}
\usetikzlibrary{backgrounds}
\usetikzlibrary{calc}
\usepackage{mdwlist}
\usepackage{multirow}
\usepackage{enumerate}
\usepackage[shortlabels]{enumitem}
\usepackage{verbatim}
\usepackage[abs]{overpic}

\title[An effective Lie--Kolchin Theorem]{An effective Lie--Kolchin Theorem for quasi-unipotent matrices}

\author{Thomas Koberda}
\address{Department of Mathematics, University of Virginia,
Charlottesville, VA 22904-4137, USA}
\email{thomas.koberda@gmail.com}
\urladdr{\href{http://faculty.virginia.edu/Koberda/}%
{http://faculty.virginia.edu/Koberda/}}

\author{Feng Luo}
\address{Department of Mathematics,
Rutgers University
Hill Center-Busch Campus
110 Frelinghuysen Road
Piscataway, NJ 08854,
USA}
\email{fluo@math.rutgers.edu}
\urladdr{\href{http://sites.math.rutgers.edu/~fluo/}%
{http://sites.math.rutgers.edu/~fluo/}}

\author{Hongbin Sun}
\address{Department of Mathematics, Rutgers University, Hill Center-Busch Campus, 110 Frelinghuysen Road, Piscataway, NJ 08854, USA}
\email{hongbin.sun@rutgers.edu}
\urladdr{\href{http://sites.math.rutgers.edu/~hs735/}%
{http://sites.math.rutgers.edu/~hs735/}}

\usepackage{enumerate}
\makeatletter
\let\@@enum@org\@@enum@
\def\@@enum@[#1]{\@@enum@org[\normalfont #1]}
\makeatother

\newtheorem{thm}{Theorem}[section]
\newtheorem{lem}[thm]{Lemma}

\newtheorem{cor}[thm]{Corollary}
\newtheorem{prop}[thm]{Proposition}

\theoremstyle{definition}
\newtheorem{defn}[thm]{Definition}

\newcommand\Z{\mathbb{Z}}
\newcommand\C{\mathbb{C}}

\newcommand\tr{\operatorname{tr}}

\newcommand\GL{\operatorname{GL}}

\newcommand\ind{\operatorname{ind}}

\newcommand\Mod{\operatorname{Mod}}

\newcommand\diag{\operatorname{diag}}

\begin{document}

\date{\today}



\begin{abstract}
We establish an effective version of the classical Lie--Kolchin Theorem. Namely, let $A,B\in\GL_m(\C)$ be quasi--unipotent matrices such that the Jordan Canonical Form of $B$ consists of a single block, and suppose that for all $k\geq0$ the matrix $AB^k$ is also quasi--unipotent.
Then $A$ and $B$ have a common eigenvector. In particular, $\langle A,B\rangle<\GL_m(\C)$ is a solvable subgroup. We give
applications of this result to the representation theory of mapping class groups of orientable surfaces.
\end{abstract}
\maketitle

\section{Introduction}
Let $V$ be a finite dimensional vector space over an algebraically closed field.
In this paper, we study the structure of certain subgroups of $\GL(V)$ which contain ``sufficiently many"
elements of a relatively simple form. We are motivated by the representation theory of the mapping class group of a surface of hyperbolic
type.

If $S$ be an orientable surface of genus $2$ or more, the \emph{mapping class group}
 $\Mod(S)$ is the group of homotopy classes
of orientation preserving homeomorphisms of $S$. The group $\Mod(S)$ is generated by certain mapping classes known as \emph{Dehn twists}, which are defined for essential simple closed curves of $S$. Here, an \emph{essential simple closed curve} is a
free homotopy class of embedded copies
of $S^1$ in $S$ which is homotopically essential, in that the homotopy class represents a nontrivial conjugacy class in $\pi_1(S)$
which is not the homotopy class of a boundary component or
a puncture of $S$.

If the genus of $S$ is
$3$ or more, it is unknown whether or not the mapping class group $\Mod(S)$ admits a faithful finite dimensional representation.
One of the main themes of the present paper is to prove
that if a representation $\rho\colon\Mod(S)\to\GL(V)$ maps a Dehn twist along a nonseparating simple closed curve to a matrix in the relatively simple form, the representation $\rho$ cannot be faithful. In the course of pursuing this thread, we prove a general result
about $2$--generated subgroups of
$\GL(V)$ which contain many quasi--unipotent elements.

\subsection{Main results}

The starting point of this paper is the following fact about images of Dehn
twists under finite dimensional linear representations of mapping class groups. This fact seems to be well-known, and appears in many
different contexts by several authors and with a number of distinct proofs. The reader may consult
Corollary 3.5 of~\cite{button2018} (cf.~\cite{button}), as well as Proposition 2.4 of~\cite{AramayonaSouto}.
We adopt the assumption here  and throughout that all vector spaces are either over the field of complex numbers
$\mathbb C$ or over an algebraically closed field of characteristic $p>0$.

\begin{prop}\label{thm:dt-intro}
Let $S$ be a surface of genus $3$ or more, let $T\in\Mod(S)$ be a Dehn twist about an essential simple closed curve on $S$, and let
$\rho\colon\Mod(S)\to\GL(V)$ be a finite dimensional linear representation of $\Mod(S)$. Then $\rho(T)$ is quasi-unipotent.
\end{prop}

Here, an element $A\in\GL(V)$ is
\emph{quasi-unipotent} if there exists a $k>0$ and $n>0$ such that the minimal polynomial of $A^k$ is $(x-1)^n$, i.e., all eigenvalues of $A$ are roots of unity. The reader will find an extensive discussion of ideas closely related to Proposition~\ref{thm:dt-intro} in~\cite{Bridson}.

The fact that Dehn twists about nonseparating curves are conjugate to each other as group elements in the mapping class group
implies that there are many elements of $\Mod(S)$ which
are sent to quasi-unipotent elements under any linear representation of $\Mod(S)$. Moreover, Dehn twists satisfy many relations amongst
each other, thus imposing further constraints on the structure of linear representations of $\Mod(S)$.

For example, if $S$ is closed with one
marked point (or equivalently from an algebraic standpoint, has a single puncture), then the Birman
Exact Sequence \cite{Farb} implies that $\Mod(S)$ contains a natural
copy of a closed surface group $\pi_1(S)$
which is generated by products of commuting Dehn twists. That is, homotopy classes of simple loops in $\pi_1(S)$ are given by products
of two twists about disjoint curves, so that under a finite dimensional linear representation $\rho$
of $\Mod(S)$, any such element of $\pi_1(S)$ is sent to a quasi-unipotent matrix. Thus, if $a,b\in\pi_1(S)$ are homotopy classes of simple
loops with geometric intersection number exactly one, then the elements $\{ab^n, a^nb, b^na, ba^n\}$ are homotopy classes of simple loops
on $S$ for all $n\in\Z$.

Motivated by this discussion of mapping class groups, we have the following
general result about linear groups which is the main result of this paper.

\begin{thm}\label{thm:main}
Let $A,B\in\GL(V)$ be quasi-unipotent matrices such that the Jordan Canonical Form of $B$ consists of a single block. Suppose that the value of
$\tr((AB^n)^k)$ is independent of $n$ for all $k$. Then the matrices $A$ and $B$ have a common eigenvector in $V$.
\end{thm}

\begin{cor}\label{cor:solvable}
Under the hypotheses of Theorem~\ref{thm:main}, there is a basis of $V$ for which the subgroup $\langle A,B\rangle$ is upper triangular,
and is hence solvable.
\end{cor}

Corollary~\ref{cor:solvable} may be compared with classical results about groups of unipotent matrices in $\GL(V)$.
The well--known theorem of Lie-Kolchin~\cite{stein} from representation theory states that if
$G<\GL(V)$ for a finite dimensional vector space $V$, and if each element $g\in G$ is
unipotent, then $G$ is a nilpotent group.

Insofar as consequences of Proposition~\ref{thm:dt-intro} and Theorem~\ref{thm:main} for mapping class groups are concerned, we note the following.

\begin{cor}\label{cor:mcg-intro1}
Let $\rho\colon\Mod(S)\to\GL(V)$ be a finite dimensional representation and $S$ be a surface of genus
 $3$ or more.

  (a) If $\rho$ is unitary, then $\rho$ is not faithful.

  (b) If the base field of $V$ has positive characteristic, then $\rho$ is not faithful.
\end{cor}

In part (a) of Corollary~\ref{cor:mcg-intro1}, we only consider the compact complex unitary groups $\mathrm{U}(n)$.
Part (b) of Corollary~\ref{cor:mcg-intro1} was first established by Button~\cite{button}.

\begin{cor}\label{cor:mcg-intro2}
Suppose $S$ is a surface with marked point and of genus at least two. If $\rho$ is a linear representation of $\Mod(S)$ and if
$\rho(T)$ has one Jordan block for a Dehn twist $T$ about a nonseparating closed curve, then
$\rho$ is not faithful.
\end{cor}

While it is true that Corollary~\ref{cor:mcg-intro2} follows from Theorem~\ref{thm:main} and Proposition~\ref{thm:dt-intro}, at least when
the genus of $S$ is sufficiently large, we can give a much more elementary proof of this fact. See Proposition~\ref{prop:mcg-easy} below.

Finally, we remark that the ideas of this paper cannot suffice by themselves to establish nonlinearity of the mapping class group, and
in particular the hypothesis in Theorem~\ref{thm:main} that the Jordan Canonical Form consists of a single block cannot be dropped.
Indeed, the work of Hadari~\cite{hadari} shows that non-virtually solvable subgroups of mapping class groups admit homological
representations with non-virtually solvable image, which in particular applies to the copy of $\pi_1(S)$ appearing in the Birman Exact
Sequence as discussed earlier.

In fact, most of the known families of representations of mapping class groups do not fall under the purview of
Corollary~\ref{cor:mcg-intro2}. The main families of naturally occurring representations which exhaust the whole mapping class group are
homological representations (arising from homology actions on finite covers of the base surface;
see~\cite{GL09,GLLM,hadari17,koberdagd,liu,looijenga,mcmullen,putwiel,sun} for instance), and
TQFT representations (arising from
certain functorial constructions; see~\cite{BHMV,RT,turaev,witten}, for instance).

In the case of homological representations, a Dehn twist is sent to a (virtual) transvection,
powers of which have matrix norm which grows linearly in the power taken. If $M$ is a unipotent
matrix with a single Jordan block of dimension at
least three, then the matrix norm of $M^n$ grows as a nonlinear polynomial in $n$ and is therefore not a transvection.

In the case of TQFT representations, Dehn twists are mapped to finite order matrices, whence some nonzero power of the  twist is sent
to the identity.  Then, the number of Jordan blocks coincides with the dimension of the representation.

\subsection{Structure of the paper}

The majority of the paper will be spent in establishing Theorem~\ref{thm:main}. First, we will reduce the result to a combinatorial statement
about the triviality of solutions to certain polynomial equations. Then, we will give a proof of the result using 
the non--negativity of certain matrices with binomial coefficients.

Structurally, we will gather relevant results about representation theory of mapping class groups first, in Section~\ref{sec:mcg}. Next,
we reduce the proof of Theorem~\ref{thm:main} to polynomial identities in Section~\ref{sec:main}, and prove the main result. 

\subsection{Remarks on motivation}
As is suggested by Corollary~\ref{cor:mcg-intro1}, much of the discussion in this paper is motivated by the problem of whether mapping
class groups are linear. This question has had a long history, and has been resolved in several cases. Notably, Bigelow~\cite{Bigelow} and
Krammer~\cite{Krammer} proved that braid groups are linear, and Bigelow--Budney~\cite{BigBud} proved mapping class groups in genus two
are linear. Building on these ideas, Korkmaz~\cite{korkmaz} proved that the hyperelliptic mapping class groups are also linear.
Mapping class groups tend to share many properties with automorphism groups of free groups, and these latter groups are known
to be nonlinear for a free group of rank $3$ or more by a result of Formanek--Procesi~\cite{FP}.
The reason for the nonlinearity of automorphism groups of free groups comes from certain
toxic subgroups which are known to be absent in mapping class groups by the work of Brendle--Hamidi-Tehrani~\cite{Brendle}.

\section{Representations of mapping class groups}\label{sec:mcg}
Before proving Theorem~\ref{thm:main}, we will detail some of its applications to the groups which inspired the present article, namely
mapping class groups of surfaces.
Throughout this section, we set $S$ to be an orientable connected surface of finite type, of genus $g\geq 3$.
As before, we write $\Mod(S)$ for the mapping class group of $S$, which is to say the group
of homotopy classes of orientation-preserving homeomorphisms of $S$. The primary purpose of this section is to prove the following fact:

\begin{prop}\label{thm:mcg-unipotent}
Let $\rho\colon \Mod(S)\to\GL_n(\C)$ be a representation of the mapping class group of a surface of genus $3$ or more, and let $T\in\Mod(S)$ be a Dehn twist. Then $\rho(T)$ is quasi-unipotent.
\end{prop}

\subsection{Centralizers and products of commutators}

The following lemma applies to general groups. An almost identical argument can be found as Lemma 2.5 in~\cite{AramayonaSouto}, in the
context of mapping class groups.

\begin{lem}\label{lem:central-comm}
Let $G$ be a group and let $g \in G$ be an element which is the product of commutators \[g=\prod_{i=1}^n [x_i, y_i]\] such that
the commutators $[x_i,g]$ and $[y_i,g]$ are trivial for all $i$.
If \[\rho\colon G\to \GL_n(\C)\] is an arbitrary
representation,  then  $\rho(g)$ is quasi-unipotent.
\end{lem}
\begin{proof}
Note that $\det\rho(g)=1$, since $g$ is a product of commutators.
If $h\in G$ is another element commuting with $g$, the  $h$ preserves each (generalized) eigenspace of $\rho(g)$. Thus, if $\lambda$
is an eigenvalue of $\rho(g)$ with generalized eigenspace $W$,
then $W$ is invariant under $\rho(x_i)$ and $\rho(y_i)$ for each $i$.
Restricting to such an eigenspace $W$, the fact that \[g=\prod_{i=1}^n [x_i, y_i]\] implies that $\det( \rho(g)|_W)=1$,
and consequently that the corresponding eigenvalue of $\rho(g)$ must be a root of unity.
Decomposing $\C^n$ as a direct sum of generalized eigenspaces of $g$, we conclude that $\rho(g)$ is quasi-unipotent.
\end{proof}

\subsection{Dehn twists}
For generalities on mapping class groups used here and in the sequel, we refer the reader to~\cite{Farb}.
Let $S$ be a surface of genus $3$ or more, and let $\gamma\subset S$ be an essential simple closed curve.
We write $T=T_{\gamma}$ for the Dehn twist about $\gamma$.

\begin{lem}\label{lem:dehn-twist}
There exist six elements \[\{x_1,x_2,x_3, y_1,y_2,y_3\} \subset \Mod(S)\]
such that \[T=\prod_{i=1}^3 [x_i, y_i],\] and such that the commutators $[T,x_i]$ and $[T,y_i]$ are trivial for all $i$.
\end{lem}
\begin{proof}
Cut the surface $S$ open along $\gamma$ to obtain a subsurface $X$ of  genus at least two. If $\gamma$ was separating, then $\gamma$ forms a boundary component of $X$. If $\gamma$
was nonseparating, then the genus of $X$ is exactly one less than the genus of $S$, and $X$  acquires two extra boundary components
from $\gamma$.

Since $X$ has genus at least two, one can embed the sphere $S_{0,4}$ in $X$ such that
exactly one of the boundary components of $X$
arising from $\gamma$ (which we will also call $\gamma$)
 is a boundary component of $S_{0,4}$, and where all other three boundary components
$\{\gamma_1, \gamma_2, \gamma_3\}$ of $S_{0,4}$ are non-separating simple loops in $X$.

Applying the lantern relation \cite{dehn,Farb}, we have \[T=T_1^{-1}B_1T_2^{-1}B_2T_3^{-1}B_3\] where
$T_i$ is the Dehn twist along $\gamma_i$ and $B_i$ is the Dehn twist along a certain simple closed loop $\beta_i$
which is essential in $S_{0,4}$ and nonseparating in $X$.
In particular, there exist orientation preserving self-homeomorphisms $\phi_i$ of $X$ sending
$\gamma_i$ to $\beta_i$ and such that the restriction $\phi_i|_{\partial X}$ is the identity.
Extending $\phi_i$ to be a self-homeomorphism of $S$,
retaining the notation $\phi_i$, we may arrange for $\phi_i$ to commute with $T$. We thus see that
\[T = [\phi_1, B_1][\phi_2, B_2][\phi_3, B_3]\] and the commutators $[T,B_i]$ and $[T, \phi_i]$ are trivial for each $i$.
\end{proof}

Proposition~\ref{thm:mcg-unipotent} is now immediate.
The following is a straightforward corollary of Lemma~\ref{lem:central-comm} and Lemma~\ref{lem:dehn-twist}:

\begin{cor}\label{cor:not-faithful}
Let $\rho\colon \Mod(S)\to\GL_n(F)$ be a representation, where  $F$ is a field. Suppose that $F$ has characteristic $p$, or suppose
that $F=\C$ and the image of $\rho$ is unitary. Then $\rho$ is not faithful.
\end{cor}

As remarked in the introduction,
Corollary~\ref{cor:not-faithful} in the case of a field with positive characteristic was obtained by J. Button \cite{button}.

\begin{proof}[Proof of Corollary~\ref{cor:not-faithful}]
In the first case, if $F$ has characteristic $p$ then unipotent matrices in $\GL_n(F)$ have finite order, so that Dehn twists about simple closed curves will
not have infinite order under $\rho$. In the second case,
compact unitary groups contain no nontrivial unipotent elements, so that quasi-unipotent unitary matrices have
finite order, so that again Dehn twists cannot have infinite order under $\rho$.
\end{proof}

\subsection{The Birman Exact Sequence}

Let $\Mod(S,p)$ denote the mapping class group of $S$ with a marked point $p$ in the interior of $S$. There is a well-known exact sequence
known as the Birman Exact Sequence given by \[1\to\pi_1(S)\to\Mod(S,p)\to\Mod(S)\to 1,\] where the map $\Mod(S,p)\to\Mod(S)$ is the map
which forgets the marked point.  The subgroup $\pi_1(S)<\Mod(S,p)$ is called the \emph{point--pushing subgroup}.

The group $\pi_1(S)$ is generated by homotopy classes of simple closed curves based at $p$. If $\gamma$ is a simple closed loop based at $p$, then the element $\gamma$ viewed as an element of $\Mod(S,p)$ is given by a product $T_{\gamma_1}T_{\gamma_2}^{-1}$ of Dehn
twists about parallel copies $\gamma_1$ and $\gamma_2$ of $\gamma$, such that $\gamma_1$ and $\gamma_2$ cobound an annulus
containing the marked point $p$. We obtain the following fact immediately from the observation that $T_{\gamma_1}$ and $T_{\gamma_2}$ commute with each other, and that $\Mod(S,p)$ acts transitively on the
set of nonseparating simple closed curves on $S$.

\begin{lem}\label{lem:birman}
Let $\rho\colon \Mod(S,p)\to\GL_n(\C)$ be a representation. The image $\rho(\pi_1(S))$ of point--pushing subgroup is generated by quasi-unipotent matrices. Moreover, if $a,b\in\pi_1(S)$ are based loops with geometric intersection number exactly one, then for all $n$ we have
that the matrices $\rho(a^nb)$, $\rho(ab^n)$, $\rho(ba^n)$, and $\rho(b^na)$ are all quasi-unipotent and conjugate in $\GL_n(\C)$.
\end{lem}

We finally obtain the following consequence of Theorem~\ref{thm:main}.

\begin{cor}\label{cor:jordan}
Let $\rho\colon \Mod(S,p)\to\GL_n(\C)$ be a representation, and suppose that for some nonseparating simple closed homotopy class of curves $a\in\pi_1(S)<\Mod(S,p)$, we have that $\rho(a)$ has a single Jordan block. Then $\rho$ is not faithful.
\end{cor}

\subsection{Representations sending a Dehn twist to a single Jordan block}

In this subsection,
suppose that $S$ has genus at least two, so that there exists a configuration of nonseparating simple closed curves $\{a,b,c\}$ such
that $a$ is disjoint from both $b$ and $c$ and where $b$ and $c$ have nonzero geometric intersection number. We write $\{T_a,T_b,T_c\}$
for the corresponding Dehn twists about these curves.

\begin{prop}\label{prop:mcg-easy}
Let $\rho\colon\Mod(S)\to\GL_n(\C)$ be a representation, and suppose that for some nonseparating simple closed curve $a\subset S$,
we have $\rho(T_a)$ consists of a single Jordan block. Then $\rho$ is not faithful.
\end{prop}
The reader may compare this fact with Corollary~\ref{cor:mcg-intro2}.

\begin{proof}[Proof of Proposition~\ref{prop:mcg-easy}]
Let $\{a,b,c\}$ be nonseparating simple closed curves such
that $a$ is disjoint from both $b$ and $c$ and where $b$ and $c$ have nonzero geometric intersection number. Then $T_a$ commutes
with both $T_b$ and $T_c$. Writing $V=\C^n$, there is a sequence of vector spaces
\[V=V_n\subset V_{n-1}\subset\cdots\subset V_1\subset  \{0\}\] such that $V_i$ has dimension exactly $i$ and such that
$\rho(T_a)(V_i)=V_i$. That is, $\rho(T_a)$ preserves a maximal flag $\mathcal{F}$ in $V$. This flag gives rise to a basis with respect to which
$\rho(T_a)$ is in Jordan Canonical Form.

A straightforward calculation shows that if $M\in\GL_n(\C)$ commutes with $\rho(T_a)$ then $M$ is upper triangular with respect to the
same basis
determined by $\mathcal{F}$. Thus, the centralizer of $\rho(T_a)$ is solvable, whence $\rho(T_b)$ and $\rho(T_c)$ generate a solvable
subgroup of $\GL_n(\C)$. Since $b$ and $c$ have nonzero geometric intersection number, the group $\langle T_b,T_c\rangle<\Mod(S)$
contains a nonabelian free group. Thus, we conclude that $\rho$ is not injective.
\end{proof}

\subsection{Irreducibility of representations}

In this final subsection, we prove the following
general fact about the representation theory of mapping class groups. Strictly speaking, it is not necessary
for the discussion in this paper, though we record it for its independent interest and because its flavor is similar to the questions addressed
in this article.

\begin{prop}\label{prop:irred}
Let $S$ be a non--sporadic surface of hyperbolic type, and suppose that $Z(\Mod(S))$ is trivial. If $\Mod(S)$ admits a faithful finite--dimensional
representation, then it admits
a faithful irreducible representation.
\end{prop}

The hypothesis that $Z(\Mod(S))$ is trivial is satisfied by most mapping class groups. For closed surfaces of genus two,
the hyperelliptic involution is central, though this is in some sense the only example of a mapping class group with nontrivial center. Here,
we say that $S$ is \emph{non--sporadic} if $S$ admits two disjoint, non--isotopic, essential, nonperipheral, simple closed curves.

We will require the following fact, which is standard (see~\cite{McPap,DGO,Ivanov1992}).

\begin{lem}\label{lem:normal}
Let $1\neq N<\Mod(S)$ be a non--central normal subgroup of a non--sporadic mapping class group.
Then $N$ contains a pair of independent pseudo-Anosov mapping classes, and hence a nonabelian free subgroup.
\end{lem}

Here, we say  that two pseudo-Anosov mapping classes are \emph{independent} if they do not generate a virtually cyclic group.

\begin{proof}[Proof of Proposition~\ref{prop:irred}]
Let $\rho\colon\Mod(S)\to\GL(V)$ be a faithful representation of minimal dimension. We claim that $\rho$ is irreducible. Suppose for a
contradiction that
$0\neq W\subsetneq V$ is a proper $\rho$--invariant subspace.
One obtains two non--injective representations of $\Mod(S)$, namely
$\rho_W\colon \Mod(S)\to\GL(W)$ and $\rho_{V/W}\colon\Mod(S)\to \GL(V/W)$. Write $K_1$ and $K_2$ for the kernels of these
two representations.

We first claim that $K=K_1\cap K_2$ is nontrivial. Indeed, otherwise the product representation $\rho_W\times \rho_{V/W}$ would be injective,
and the subgroup $K_1K_2$ would be isomorphic to $K_1\times K_2$. By Lemma~\ref{lem:normal}, both $K_1$ and $K_2$ contain
pseudo-Anosov mapping classes, whose centralizers are virtually cyclic. Then if $K_1$ and $K_2$ are both nontrivial, we obtain that the
centralizer of each infinite element $\psi\in K_i$ contains a copy of $\Z\times\Z$,
which is a contradiction. It follows that either $K_1$ or $K_2$ is trivial,
violating the minimality of the dimension of $V$.

We may thus conclude that $K$ is nontrivial, and that $\rho$ restricts to a faithful representation of $K$. If $1\neq k\in K$
then we immediately have
that $v-\rho(k)v\in W$ for all $v\in V$, since $k\in K_2$. Moreover, if $w\in W$ then $\rho(k)w=w$,
since $k\in K_1$. It follows that there is a basis for $V$ for which
$\rho(K)$ acts by unipotent matrices, so that $\rho(K)$ is nilpotent. This contradicts Lemma~\ref{lem:normal}.
\end{proof}

\section{Trace calculations}\label{sec:main}

In this section, let $A,B\in\GL_{m+1}(\C)$
be quasi-unipotent matrices such that $B$ has a single Jordan block, and suppose that the value of $\tr((AB^n)^k)$ is
independent of the value of $n$ for all $k$. The choice of $\GL_{m+1}(\C)$ instead of $\GL_m(\C)$ is for notational convenience later on.

\subsection{An expansion of $(AB^n)^k$}
Since $B$ is quasi-unipotent, we may replace $B$ by a positive power which is a unipotent matrix.
We will abuse notation and call this power $B$ as well. In this case, we conjugate so that we
 may write $B=I+N$, where $N$ is a nilpotent matrix.

We may therefore write \[AB^n=\sum_{i=0}^n \binom{n}{i}AN^i,\] and
\[(AB^n)^k=\sum_{j=0}^{kn}\sum_{i_1+\cdots i_{k}=j, i_s \geq 0}\binom{n}{i_1}\cdots\binom{n}{i_{k}}AN^{i_1}\cdots AN^{i_{k}},\] and the assumptions
on $A$ and $B$ imply that the trace of this latter expression is independent of $n$.

Since $N^{m+1}=0$, we may write the limits of the sum as quantities which are independent of $n$. That is, we have
\[(AB^n)^k=\sum_{j=0}^{km}\sum_{i_1+\cdots i_{k}=j}\binom{n}{i_1}\cdots\binom{n}{i_{k}}AN^{i_1}\cdots AN^{i_{k}}.\]

\subsection{The index of a matrix}
Let $A=(A_{i,j})\in M_{m+1}(\C)$ be a square matrix.

\begin{defn}
The \emph{index} of the matrix $A$ is the unique integer $k\in\Z$ characterized by the following two conditions:
\begin{enumerate}
\item
If $i>j+k$ then $A_{i,j}=0$.
\item
There exists an $i$ such that $A_{i,i-k}\neq 0$.
\end{enumerate}
The index of $A$ will be written $\ind(A)$.
\end{defn}

It is a straightforward consequence of the definition that if $\ind(A)=k$ and if $A_{i,j}\neq 0$ then $j\geq i-k$.

The following lemma is straightforward, and we omit a proof.

\begin{lem}\label{lem:index-basic}
Let $A=(A_{i,j})\in M_{m+1}(\C)$. The following conclusions hold:
\begin{enumerate}
\item
If $A_{i,j}=0$ whenever $i>j+k$ then $\ind(A)\leq k$.
\item
We have $-m-1\leq \ind(A)\leq m$.
\item
The matrix $A$ is upper triangular if and only if $\ind(A)\leq 0$.
\item
If $\ind(A)<0$ then $\tr(A)=0$.
\end{enumerate}
\end{lem}

The following is a fundamental property of the index:

\begin{lem}\label{lem:submult}
The index is submultiplicative. That is, for square matrices $A$ and $B$ of the same dimension, we have \[\ind(AB)\leq \ind(A)+\ind(B).\]
\end{lem}
\begin{proof}
We write $\ind(A)=k$ and $\ind(B)=\ell$. By Lemma~\ref{lem:index-basic}, it suffices to show that if $i>j+k+\ell$ then $(AB)_{i,j}=0$.

Suppose the contrary. Then for some such choice of indices $i$ and $j$ such that $i>j+k+\ell$,
we have \[0\neq (AB)_{i,j}=\sum_s A_{i,s}B_{s,j}.\]
We choose an index $s$ such that $A_{i,s}B_{s,j}\neq 0$. Since the index of
$A$ is equal to $k$, we see that $s\geq i-k$, and similarly $j\geq s-\ell$. Combining these equalities, we see that
\[j\geq s-\ell\geq i-k-\ell,\] which violates the requirement that $i>j+k+\ell$.
\end{proof}

The following properties of the index now follow:

\begin{cor}\label{cor:prod-index}
Let $\{A_1,\ldots,A_s\}\subset M_{m+1}(\C)$, and write $\ind(A_i)=n_i$. Let \[K=\sum_{i=1}^s n_i\] and let $A=A_1\cdots A_s$.
\begin{enumerate}
\item
If $K\leq 0$ then $A$ is upper triangular.
\item
If $K<0$ then $\tr(A)=0$.
\item
If $K=0$ then \[\tr(A)=\sum_{k=1}^{m+1}(A_1)_{k,k-n_1}(A_2)_{k-n_1,k-n_1-n_2}\cdots (A_s)_{k-n_1-\cdots-n_{s-1},k}.\]
\end{enumerate}
\end{cor}

By convention, $(A_i)_{j,k}=0$ if one of $j$ and $k$ is nonpositive.

\begin{proof}[Proof of Corollary~\ref{cor:prod-index}]
We establish the conclusions in order. The first and second conclusions follow immediately from Lemma~\ref{lem:submult}, by induction on $s$.

For the third conclusion, we note the general formula \[\tr(A)=\sum_{i_1,\ldots,i_s} (A_1)_{i_1,i_2}(A_2)_{i_2,i_3}\cdots(A_s)_{i_s,i_1},\] where
the indices $\{i_1,\ldots,i_s\}$ range between $1$ and $m+1$. Now, if $(A_k)_{i_k,i_{k+1}}\neq 0$ then $i_{k+1}\geq i_k-n_k$, as follows
from the definition of the index. Here, we adopt the convention that $i_{s+1}=i_1$. The only terms in the expression
for $\tr(A)$ which are nonzero are ones for which $i_{k+1}\geq i_k-n_k$ for each index $k$. From the fact that $K=0$,
the conclusion of the lemma follows.
\end{proof}

\subsection{A combinatorial reformulation of Theorem~\ref{thm:main}}
Let $A\in M_{m+1}(\C)$ have index $r$. We will write $x_i=A_{i,i-r}$ for $r+1\leq i\leq m+1$.
Our goal is to show that $x_i=0$ using the conditions that $\tr((AB^n)^k)$ are independent of $n$, for a fixed $k$.  By applying this argument inductively on the index of $A$, it shows that $A$ is upper-triangular. Therefore the group generated by $A,B$ is solvable.
We write $N$ for the nilpotent matrix obtained from
a single $(m+1)\times (m+1)$ Jordan block, with $N^{m+1}=0$ holds. That is, if $B\in M_{m+1}(\C)$ is a matrix with ones down the diagonal and the upper off-diagonal,
then $N=B-I$.

Note that by Lemma~\ref{lem:submult}, we have \[\ind(AN^{i_1}\cdots AN^{i_{k}})\leq r\cdot k-\sum_{\ell=1}^{k} i_{\ell}.\]

We consider the trace of the expansion of $(AB^n)^k$ in terms of $N$, which is to say
\[\sum_{j=0}^{km}\sum_{i_1+\cdots i_{k}=j}\binom{n}{i_1}\cdots\binom{n}{i_{k}}\tr(AN^{i_1}\cdots AN^{i_{k}}).\] Here,
the $i_k$'s are non-negative integers.

By Lemma \ref{lem:index-basic}, we have that if \[j=\sum_{\ell=1}^k i_{\ell}>r\cdot k\] then the corresponding summand contributes zero.
Thus, the largest value of $j$ for which the trace of $AN^{i_1}\cdots AN^{i_{k}}$ is
nonzero is when $j=r\cdot k$, and we will focus on this term. Thus, viewing this sum
of traces as a function of $n$, the first powers of $n$ with nonzero coefficient come from terms in which \[\sum_{\ell=1}^k i_{\ell}=r\cdot k.\] It is easy
to see that we obtain a polynomial function of $n$, that this polynomial has degree at most $r\cdot k$, and that the highest degree terms
come from the summands for which \[\sum_{\ell=1}^k i_{\ell}=r\cdot k.\]

We can also consider the product of binomial coefficients
\[\binom{n}{i_1}\cdots\binom{n}{i_{k}}\] as a function of $n$. This function is again a polynomial,
and the coefficient of the highest degree term in $n$ is given by \[\frac{1}{i_1!\cdots i_k!}.\]

Now, if we consider the $s^{th}$ power $N^s$ of $N$, we have that $(N^s)_{i,i+s}=1$ and $(N^s)_{i,t}=0$ otherwise. Applying Corollary
~\ref{cor:prod-index}, we can easily compute
\[\tr(AN^{i_1}\cdots AN^{i_{k}})=\sum_{s=1}^{m+1} x_sx_{s-r+i_1}\cdots x_{s-(k-1)r+ i_1+\cdots+ i_{k-1}}.\]
Here, we adopt the convention that if $i\leq r$ or $i\geq m+2$, then $x_i=0$.

Thus, the contribution of the traces for which \[\sum_{\ell=1}^k i_{\ell}=r\cdot k\] is a polynomial of $n$, whose highest degree term is given
by \[\sum_{i_1+\cdots+i_k=r\cdot k}\left(\sum_{s=1}^{m+1}\frac{1}{i_1!\cdots i_k!}x_sx_{s-r+i+1}\cdots x_{s-(k-1)r+i_1+\cdots+i_{k-1}}\right).\]

Thus, if
\[\sum_{j=0}^{km}\sum_{i_1+\cdots i_{k}=j}\binom{n}{i_1}\cdots\binom{n}{i_{k}}\tr(AN^{i_1}\cdots AN^{i_{k}}),\] independently of $n$, then
we obtain
 \[\sum_{i_1+\cdots+i_k=r\cdot k}\left(\sum_{s=1}^{m+1}\frac{1}{i_1!\cdots i_k!}x_sx_{s-r+i_1}\cdots x_{s-(k-1)r+i_1+\cdots+i_{k-1}}\right)=0,\]
 independently of $n$.

 To simplify these expressions a little, we substitute $j_s=i_s-r$. This way, we obtain the following  lemma.

 \begin{lem}\label{lem:comb}
 Theorem~\ref{thm:main} is equivalent to the statement that if
 \[\sum_{j_1+\cdots+j_k=0,\, j_t\geq -r}\left(\sum_{s=1}^{m+1}\frac{1}{(j_1+r)!\cdots(j_k+r)!}x_sx_{s+j_1}\cdots x_{s+j_1+\cdots+j_{k-1}}\right)=0\]
 for all positive integers $k$, then $x_1=\cdots=x_{m+1}=0$.
 \end{lem}

 Lemma~\ref{lem:comb} can be refined further.

 \begin{lem}\label{lem:comb2}
 Let $\{x_1,\ldots,x_{m+1}\}$ be complex numbers such that
  \[\sum_{j_1+\cdots+j_k=0,\, j_t\geq -r}\left(\sum_{s=1}^{m+1}\frac{1}{(j_1+r)!\cdots(j_k+r)!}x_sx_{s+j_1}\cdots x_{s+j_1+\cdots+j_{k-1}}\right)=0\]
 for all positive integers $k$.
 Then the numbers $\{x_1,\ldots,x_{m+1}\}$ satisfy the equation
 \[\sum_{i_1,\ldots,i_k=1}^{m+1}\frac{x_{i_1}\cdots x_{i_k}}{(i_2-i_1+r)!(i_3-i_2+r)!\cdots(i_k-i_{k-1}+r)!(i_1-i_k+r)!}=0,\] where the indices of
 $\{i_1,\ldots,i_k\}$ lie in $\{1,\ldots,m+1\}$, and where if $t<0$ we adopt the convention $1/t!=0$.
 \end{lem}


 \begin{proof}
 Let $s$ by as in Lemma~\ref{lem:comb}. We set \[i_1=s,i_2=s+j_1,\ldots,i_k=s+j_1+\cdots+j_{k-1}.\] We thus get $j_n=i_{n+1}-i_n$ for
 $n\in\{1,\ldots,k-1\}$ and \[j_k=i_1-i_k=-(j_1+j_2+\cdots+j_{k-1}).\] The claim of this lemma follows immediately from these substitutions.
 \end{proof}

\subsection{Totally nonnegative matrices}
For notational convenience, we write  \[p_k=\sum_{i_1,\ldots,i_k=1}^{m+1}\frac{x_{i_1}\cdots x_{i_k}}{(i_2-i_1+r)!(i_3-i_2+r)!\cdots(i_k-i_{k-1}+r)!(i_1-i_k+r)!},\] where again
by convention we set \[\frac{1}{n!}=0\] whenever $n<0$.
In this section, we prove the following result which implies, by Lemmas \ref{lem:comb} and \ref{lem:comb2}, that Theorem \ref{thm:main} holds.

\begin{thm}\label{thm:pk}
Suppose $p_k(x_1,\ldots,x_{m+1})=0$ for all $k$. Then $x_i=0$ for $1\leq i\leq m+1$.
\end{thm}

Theorem~\ref{thm:pk} admits a quick reduction to a statement about matrices with combinatorial quantities as entries, which we carry out
here before completing the proof.

First, we perform a change of variables, setting $x_i=y_i^2$. Then the statement $p_k(x_1,\ldots,x_{m+1})=0$ becomes
\[\sum_{i_1,\ldots,i_k=1}^{m+1}\frac{(y_{i_1}y_{i_2})(y_{i_2}y_{i_3})\cdots (y_{i_k}y_{i_1})}{(i_2-i_1+r)!(i_3-i_2+r)!
\cdots(i_k-i_{k-1}+r)!(i_1-i_k+r)!}=0.\]

We let \[a_{i,j}=\frac{y_iy_j}{(i-j+r)!},\] and set $\mathfrak{A}=(a_{i,j})$. Then the equation $p_k=0$ for all $k$ is merely the equation
$\tr\mathfrak{A}^k=0$ for all $k$. In particular, the matrix $\mathfrak{A}$ is nilpotent. It follows that $\mathfrak{A}^{m+1}=0$, and so that in particular we have
$\det\mathfrak{A}=0$.

We now set \[b_{i,j}=\frac{1}{(i-j+r)!}\] and set $\mathfrak{B}=(b_{i,j})$. Note that
\[\mathfrak{A}=\diag(y_1,\ldots,y_{m+1})\cdot\mathfrak{B}\cdot\diag(y_1,\ldots,y_{m+1}),\]
where $\diag(y_1,\ldots,y_{m+1})$ denotes a diagonal matrix with the
corresponding entries. It follows that \[\det\mathfrak{A}=\det\mathfrak{B}\cdot\prod_{i=1}^{m+1} x_i.\] Note that if $\det\mathfrak{B}\neq 0$ then
$x_i=0$ for some $i$. Thus, a straightforward induction on $m$ shows that the following result implies Theorem~\ref{thm:pk}:
First of all, $\det\mathfrak{A}=0$ implies that $y_i=0$ for at least one index $i$.
By plugging $y_i=0$ into $\mathfrak{A}$, we get that the $i^{th}$ row and $i^{th}$ column of $\mathfrak{A}$ consist of only zeros. Once we delete the $i^{th}$ row and $i^{th}$ column of $\mathfrak{A}$, we get an $m\times m$ matrix \[\mathfrak{D}=\diag(y_1,\ldots,\widehat{y_i},\ldots,y_{m+1})\cdot\mathfrak{C}\cdot\diag(y_1,\ldots,\widehat{y_i},\ldots,y_{m+1}),\] 
where $\mathfrak{C}$ is the $i^{th}$ principal minor of  $\mathfrak{B}$ of size $m$. Then the fact that $\mathfrak{A}$ is nilpotent implies that $\mathfrak{D}$ is nilpotent. So if the principal minors of $\mathfrak{B}$ are all nonsingular, we obtain that one of
\[\{y_1,\ldots,\widehat{y_i},\ldots,y_{m+1}\}\]
equals $0$, and the induction process can proceed.

We therefore need only establish the following result.

\begin{thm}\label{thm:positive}
The determinants of all principal minors of $\mathfrak{B}$ are nonzero.
\end{thm}

Note that if $M$ is a matrix whose principal minors are all nonsingular, then if we multiply $M$ by a nonsingular diagonal matrix
(on the left or on the right), the principal
minors of the resulting matrix remain nonsingular. Thus, we may modify $\mathfrak{B}$ by multiplying by nonsingular diagonal matrices
in order to convert it into a more advantageous form, without affecting the statement of Theorem~\ref{thm:positive}. With this in mind,
we multiply on the left by \[\diag((r+1)!,(r+2)!,\ldots,(r+m+1)!),\] and on the right by \[\diag(1/1!,1/2!,\ldots,1/(m+1)!).\] The resulting matrix will
be denoted $M(r,m)=(f_{i,j})$, where we have \[f_{i,j}=\binom{i+r}{j}=\frac{(i+r)!}{j!(i-j+r)!}.\]

Thus, it suffices to prove the following result:

\begin{thm}\label{thm:positive-m}
The determinants of all principal minors of $M$ are nonzero.
\end{thm}

We will prove Theorem~\ref{thm:positive-m} by establishing the following fact.

\begin{lem}\label{lem:positive-m}
Let \[0<q_1<\cdots<q_m\] and $r\geq 0$ be integers, and let $p_k=q_k+r$. Let $M$ be a matrix whose entries are given by
\[f_{i,j}=\binom{p_i-1}{q_j-1}.\] Then the determinant of $M$ is positive.
\end{lem}

We retain a standing convention that if $p<q$ then \[\binom{p}{q}=0.\] It is clear that Lemma~\ref{lem:positive-m} implies
Theorem~\ref{thm:positive-m}.

A matrix is called \emph{totally nonnegative} if all of its minors have nonnegative determinant. For a square $n\times n$ matrix $A$, we set
$I=\{i_1,\ldots,i_m\}$ and $J=\{j_1,\ldots,j_m\}$ to be subsets of $[n]=\{1,\ldots,n\}$ which have the same size. We write $A_{I,J}$ for the
minor of $A$ whose row indices lie in $I$ and whose column indices lie in $J$.

The determinant of a minor of a product of two matrices can be expanded from minors of the two factor matrices. Specifically, we have
the following fact, classically known as the Binet--Cauchy formula \cite{prasolov}.

\begin{lem}\label{lem:cauchy}
Let $A$ and $B$ be $n\times n$ matrices, and let $I,J\subset [n]$ have cardinality $m$. Then \[\det((AB)_{I,J})=
\sum_{K\subset [n],\, |K|=m} \det(A_{I,K})\det(B_{K,J}).\]
\end{lem}

An immediate consequence of the Cauchy--Binet formula is that the product of two totally nonnegative matrices is again totally nonnegative.
We define the \emph{lower triangular Pascal matrix} $L_n$ (see \cite{ADGP}) to be the $n\times n$ matrix whose entries are given
by \[(L_n)_{i,j}=\binom{i-1}{j-1}.\] With our convention, it becomes clear that if
$i<j$ then the corresponding entry of $L_n$ is zero, so that $L_n$ is indeed lower triangular. Observe that if $n>p_m$ then the matrix $M$
as defined in Lemma~\ref{lem:positive-m} is a minor of $L_n$.

\begin{lem}\label{lem:pascal}
For $n\geq 1$, the matrix $L_n$ is totally nonnegative.
\end{lem}
\begin{proof}
Let $E_{i,j}$ be a matrix whose unique nonzero entry is in the $(i,j)$ entry. A direct inductive computation
(cf. Lemma 1 of \cite{ADGP}) yields
\[L_n=(I_n+E_{n,n-1})(I_n+E_{n,n-1}+E_{n-1,n-2}))\cdots(I_n+E_{n,n-1}+E_{n-1,n-2}+\cdots+E_{2,1}).\]
We can further expand the factors in this product as
\[I_n+E_{n,n-1}+E_{n-1,n-2}+\cdots+E_{i,i-1}=(I_n+E_{i,i-1})(I_n+E_{i+1,i})\cdots (I_n+E_{n,n-1}).\]
Thus to establish the lemma, it suffices to show that each matrix of the form \[I_n+E_{i,i-1}\] is totally nonnegative. To do this, we compute
the determinant of \[(I_n+E_{i,i-1})_{I,J}.\] If $i\notin I$ then the only possibility for $J$ for which $(I_n+E_{i,i-1})_{I,J}$ has nonzero determinant
is for $I=J$, wherein the determinant is $1$. If $i\in I$, then $(I_n+E_{i,i-1})_{I,J}$ has nonzero determinant only if $I=J$ or if
\[J=(I\setminus\{i\})
\cup\{i-1\}\] and $i-1\notin J$. It is straightforward to see then that the determinant is $1$ in both these cases. The conclusion of the lemma
follows.
\end{proof}

It follows then that $M$ is totally nonnegative, so that $\det M\geq 0$.

\begin{proof}[Proof of Lemma~\ref{lem:positive-m}]
It suffices to show that $\det M>0$. Let $P=\{p_1,\ldots,p_m\}$ and let $Q=\{q_1,\ldots, q_m\}$ with $p_k=q_k+r$. It suffices to show that $\det((L_n)_{P,Q})>0$.
This determinant can be computed from the Cauchy--Binet formula. Indeed, we may expand $\det((L_n)_{P,Q})$ as
\[\sum_{R_2,\ldots,R_{n-1}\subset [n],\, |R_i|=m}\det((I_n+E_{n,n-1})_{P,R_{n-1}})\cdots
\det((I_n+E_{n,n-1}+\cdots+E_{2,1})_{R_2,Q}).\] Note that each term in this sum is nonnegative, so that we need only find suitable
$\{R_2,\ldots,R_{n-1}\}$ so that each of the corresponding minors is positive. We will adopt the notation $P=R_n$ and $Q=R_1$.

We observe that \[\det((I_n+E_{n,n-1}+\cdots+E_{k,k-1})_{R_k,R_{k-1}})=1\] if and only if there is a (possibly empty) subset
\[\{a_1,a_2,\ldots,a_j\}\subset R_k\cap\{k,\ldots,n\}\] satisfying \[R_{k-1}=(R_k\setminus\{a_1,\ldots,a_j\})\cup\{a_1-1,\ldots,a_j-1\},\]
such that they have the same cardinality. A proof of
this fact can be given by an argument identical to that given in the proof of Lemma~\ref{lem:pascal}.

For the matrix $M$ under consideration, we have \[R_n=\{p_1,\ldots,p_m\}=\{q_1+r,\ldots,q_m+r\},\] and $R_1=\{q_1,\ldots, q_m\}$. We set
\[R_n=R_{n-1}=\cdots=R_{q_1+r}=\{q_1+r,\cdots,q_m+r\},\] then write \[R_{q_1+r-j}=\{q_1+r-j,q_2+r-j,\ldots,q_m+r-j\}\] for $1\leq j\leq r$, and finally write \[R_{q_1-1}=\cdots=R_1=\{q_1,q_2,\ldots,q_m\}.\]
This exhibits a suitable choice of $\{R_1,\ldots,R_n\}$ and hence proves that $\det M>0$.
\end{proof}

\section*{Acknowledgements}
The first author is partially supported by an Alfred P. Sloan Foundation Research Fellowship and by NSF Grant DMS-1711488.  The second author is partially supported by NSF Grants DMS-1760527, DMS-1737876 and DMS-1811878.
The third author is partially supported by NSF Grant DMS-1840696. The authors are grateful to A. Hadari for helpful discussions and to the
anonymous referee for several comments which improved the paper.

\end{document}